\documentclass[12pt]{article}%
\usepackage{amsmath}
\usepackage{amsfonts}
\usepackage{amssymb}
\usepackage{graphicx}%
\providecommand{\U}[1]{\protect\rule{.1in}{.1in}}
\newtheorem{theorem}{Theorem}

\newtheorem{corollary}[theorem]{Corollary}

\newtheorem{example}[theorem]{Example}

\newtheorem{lemma}[theorem]{Lemma}

\newenvironment{proof}[1][Proof]{\noindent\textbf{#1.} }{\ \rule{0.5em}{0.5em}}
\begin{document}

\title{The complete set of efficient vectors for a pairwise comparison matrix }
\author{Susana Furtado\thanks{Email: sbf@fep.up.pt. The work of this author was
supported by FCT- Funda\c{c}\~{a}o para a Ci\^{e}ncia e Tecnologia, under
project UIDB/04561/2020.} \thanks{Corresponding author.}\\CMAFcIO and Faculdade de Economia \\Universidade do Porto\\Rua Dr. Roberto Frias\\4200-464 Porto, Portugal
\and Charles R. Johnson \thanks{Email: crjohn@wm.edu. }\\Department of Mathematics\\College of William and Mary\\Williamsburg, VA 23187-8795.}
\maketitle

\begin{abstract}
Efficient vectors are the natural set from which to choose a cardinal ranking
vector for a pairwise comparison matrix. Such vectors are the key to certain
business project selection models. Many ways to construct specific efficient
vectors have been proposed. Yet, no previous method to produce all efficient
vectors was known. Here, using some graph theoretic ideas, as well as a
numerical extension technique, we show how to generate inductively all
efficient vectors for any given pairwise comparison matrix. We apply this
method to give a matricial proof of the fact that the set of efficient vectors
and other related sets are piecewise linearly connected. In addition, we
determine explicitly all efficient vectors for a $4$-by-$4$ pairwise
comparison matrix. Several examples are provided.

\end{abstract}



\bigskip

\section{Introduction\label{secintr}}

\bigskip An $n$-by-$n$ entry-wise positive matrix $A=\left[  a_{ij}\right]  $
is called a \emph{pairwise comparison} matrix or a \emph{reciprocal} matrix
if, for all $1\leq i,j\leq n$, $a_{ji}=\frac{1}{a_{ij}}.$ Such a matrix
results from independent, pair-wise ratio comparisons among $n$ alternatives.
A (positive) cardinal ranking vector $w$ deduced from $A$ may be used to
assess the relative value of the alternatives, or integrated into a broader
business decision model to rank projects \cite{baj, choo, dij, golany}. Matrix
$A$ is said to be \emph{consistent} if $a_{ij}a_{jk}=a_{ik}$ for all $i,j,k.$
This is the case if and only if there is a positive vector $w=\left[
\begin{array}
[c]{ccc}%
w_{1} & \ldots & w_{n}%
\end{array}
\right]  ^{T}$ such that $a_{ij}=\frac{w_{i}}{w_{j}},$ and such a vector is
unique up to a factor of scale \cite{saaty1977}. However, consistency of the
ratio comparisons is unlikely. This suggests choosing $w$ from among vectors
"nearly" as good. A positive vector $w$ is called \emph{efficient} for $A$
\cite{blanq2006} if, for every other positive vector $v=\left[
\begin{array}
[c]{ccc}%
v_{1} & \ldots & v_{n}%
\end{array}
\right]  ^{T},$
\[
\left\vert A-\left[  \frac{v_{i}}{v_{j}}\right]  \right\vert \leq\left\vert
A-\left[  \frac{w_{i}}{w_{j}}\right]  \right\vert
\]
(entry-wise) implies $v$ is proportional to $w,$ i.e. no other consistent
matrix is clearly better than that associated with $w$ (Pareto optimality). It
is natural to choose a cardinal ranking vector from among efficient ones.
However, when $A$ is not consistent, there are infinitely many efficient
vectors for $A,$ and many ways of constructing efficient vectors have been
proposed. For example, the (Hadamard) geometric mean of all the columns of a
reciprocal matrix has been shown to be efficient \cite{blanq2006}, and, more
generally, any weighted geometric mean of the columns is efficient (in
particular any column) \cite{FJ1, FJ3}. The efficiency of the (right) Perron
eigenvector of a reciprocal matrix, the classical proposal for the ranking
vector \cite{saaty1977, Saaty}, has been studied. It is known that this vector
may not be efficient \cite{blanq2006, bozoki2014}. Classes of matrices for
which the Perron eigenvector is efficient have been identified \cite{p6, p2,
FerFur}. Generalizing these results, all the efficient vectors in some of
these classes have been described \cite{CFF, Fu22}. Several other aspects of
efficiency have been studied (see, for example, \cite{anh, baj, european}).
But, it is an open question what all the efficient vectors are for any given
reciprocal matrix \cite{p2}. This, of course, is central to the use of
reciprocal matrices.

There is, nonetheless, a way to recognize a given vector $w$ as efficient for
$A,$ based upon a certain directed graph (digraph) $G(A,w).$ This graph has a
directed edge from $i$ to $j$, $i\neq j$, if and only if $\frac{w_{i}}{w_{j}%
}-a_{ij}\geq0.$ Vector $w$ is efficient for $A$ if and only if $G(A,w)$ is
strongly connected. This does not, by itself, identify the set of all
efficient vectors for $A,$ which we denote by $\mathcal{E}(A).$ However, we
develop here a way to generate all of $\mathcal{E}(A).$

Notice that $G(A,w)$ has at least one edge in one direction between $i$ and
$j,$ $i\neq j$ (it may have both if $w_{i}=a_{ij}w_{j}$). Such digraph is
called \emph{semi-complete}. Of course, an induced subgraph of a semi-complete
digraph is also semi-complete.

It is a natural question when an entry-deleted sub-vector of an efficient
vector for $A$ is efficient for the corresponding principal submatrix of $A.$
In general this need not happen, as will be seen in some of our examples.
However, here we explain how this happens, and this is partly the basis of how
we inductively generate all efficient vectors for a given reciprocal matrix.

Here are the advances we offer. Using the fact that each semi-complete,
strongly connected digraph on $n\geq4$ vertices has at least $2$
(semi-complete) strongly connected induced subgraphs on $n-1$ vertices, we
observe that there is an entry-deleted sub-vector of an efficient vector that
is efficient for the corresponding principal submatrix. This observation then
allows us to go in reverse from smaller matrices/vectors to larger ones. If a
vertex is added to a semi-complete, strongly connected digraph and the vertex
is then connected to each prior vertex in the appropriate direction, we may
assess which resulting semi-complete digraphs are strongly connected. When
this happens, we show how to extend a vector $w$ for which $G(A,w)$ is the
former graph to a vector $w^{\prime}$ for which $G(A^{\prime},w^{\prime})$ is
the latter graph for a given reciprocal matrix $A^{\prime}$ of which $A\ $is
the appropriate principal submatrix. By considering the right possibilities,
we show that all efficient vectors for a given reciprocal matrix may be
generated. We use this result to present a matricial proof of the fact that
the set of efficient vectors for a reciprocal matrix, as well as some other
related sets, are piecewise linearly connected. The connectivity of
$\mathcal{E}(A)$ was stated in \cite{blanq2006} (see also \cite{carrizoza1995}%
) using a general result from optimization.An application of the developed
method to generate the efficient vectors to the case of $4$-by-$4$ reciprocal
matrices is provided.Examples are given, including some explicit cases.

In the next section we review some necessary background on efficient vectors
and make some useful new observations. Then, in Section \ref{s4}, we describe
the graph theoretic result that underlay our ideas. In Section \ref{s44} we
give our method to generate all elements of $\mathcal{E}(A)$ for any
reciprocal matrix $A$ and use it in Section \ref{s5} to study the connectivity
of $\mathcal{E}(A)$ and other sets related to it. In Section \ref{sec4x4} we
apply our method to generate all efficient vectors for any $4$-by-$4$
reciprocal matrix. We conclude the paper with some observations in Section
\ref{s6}.

\section{Background}

In this section we introduce some notation and give some known results that
will be helpful in the paper.

We denote by $\mathcal{PC}_{n}$ the set of $n$-by-$n$ reciprocal matrices and
by $\mathcal{V}_{n}$ the set of positive $n$-vectors. Given $w\in
\mathcal{V}_{n},$ we denote by $w_{i}$ the $i$th entry of $w.$

For an $n$-by-$n$ matrix $A=[a_{ij}],$ the principal submatrix of $A$
determined by deleting (by retaining) the rows and columns indexed by a subset
$K\subset\{1,\ldots,n\}$ is denoted by $A(K)$ $(A[K]);$ we abbreviate
$A(\{i\})$ as $A(i).$ Similarly, if $w\in\mathcal{V}_{n},$ we denote by $w(K)$
($w[K]$) the vector obtained from $w$ by deleting (by retaining) the entries
indexed by $K$ and abbreviate $w(\{i\})$ as $w(i)$. Note that, if $A$ is
reciprocal (consistent) then so are $A(K)$ and $A[K].$

\bigskip

Next, we give a helpful result that concerns how $\mathcal{E}(A)$ changes when
$A$ is subjected to either a positive diagonal similarity or a permutation
similarity, or both (a monomial similarity) \textrm{\cite{Fu22}}.

\begin{lemma}
\label{lsim} Suppose that $A\in\mathcal{PC}_{n}$ and $w\in\mathcal{E}(A).$ If
$D$ is a positive diagonal matrix ($P$ is a permutation matrix), then
$DAD^{-1}\in\mathcal{PC}_{n}$ and $Dw\in\mathcal{E}(DAD^{-1})$ ($PAP^{T}%
\in\mathcal{PC}_{n}$ and $Pw\in\mathcal{E}(PAP^{T})$).
\end{lemma}

\bigskip In \cite{blanq2006} the authors proved a useful result that gives a
characterization of efficiency in terms of a certain digraph (see also
\cite{FJ3}). Given $A\in\mathcal{PC}_{n}$ and a positive vector $w=\left[
\begin{array}
[c]{ccc}%
w_{1} & \cdots & w_{n}%
\end{array}
\right]  ^{T}$, define $G(A,w)$ as the digraph with vertex set $\{1,\ldots
,n\}$ and a directed edge $i\rightarrow j,$ $i\neq j,$ if and only if,
$\frac{w_{i}}{w_{j}}\geq a_{ij}$.

\begin{theorem}
\textrm{\cite{blanq2006}}\label{blanq} Let $A\in\mathcal{PC}_{n}$ and
$w\in\mathcal{V}_{n}.$ The vector $w$ is efficient for $A$ if and only if
$G(A,w)$ is a strongly connected digraph, that is, for all pairs of vertices
$i,j,$ with $i\neq j,$ there is a directed path from $i$ to $j$ in $G(A,w)$.
\end{theorem}

\bigskip As mentioned in Section \ref{secintr}, $G(A,w)$ is a semi-complete
digraph (that is, has one edge between any two distinct vertices in at least
one direction).

\bigskip

Any matrix in $\mathcal{PC}_{2}$ is consistent and the efficient vectors are
the positive multiples of any of its columns. In \cite{CFF} the description of
the efficient vectors for a reciprocal matrix obtained from a consistent
matrix by modifying one entry above the main diagonal, and the corresponding
reciprocal entry below the main diagonal, was given. Such matrices were called
in \cite{p6} \emph{simple perturbed consistent matrices}. Later, in
\cite{Fu22}, an answer to the problem when $A$ is a reciprocal matrix obtained
from a consistent one by modifying at most two entries above de main diagonal,
and the corresponding reciprocal entries, was presented.

\bigskip

If $A\in\mathcal{PC}_{n}$ is a simple perturbed consistent matrix, there is a
positive diagonal matrix $D$ such that $A=DS_{n,k,l}(x)D^{-1},$ for some $x>0$
and some $1\leq k<l\leq n.$ Here, $S_{n,k,l}(x)$ denotes the reciprocal matrix
with all entries equal to $1$ except those in positions $(k,l)$ and $(l,k),$
which are $x$ and $\frac{1}{x},$ respectively. By Lemma \ref{lsim}, $w$ is an
efficient vector for $S_{n,k,l}(x)$ if and only if $Dw$ is an efficient vector
for $A.$

\begin{theorem}
\cite{CFF}\label{tmain} Let $n\geq3,$ $x>0$ and $D=\operatorname*{diag}%
(d_{1},\ldots,d_{n})$ be a positive diagonal matrix$.$ Let $A=DS_{n,k,l}%
(x)D^{-1},$ $k<l.$ Then, a vector $w\in\mathcal{V}_{n}$ is efficient for $A$
if and only if
\[
\frac{w_{l}}{d_{l}}\leq\frac{w_{i}}{d_{i}}\leq\frac{w_{k}}{d_{k}}\leq
\frac{w_{l}}{d_{l}}x,\text{ for }i=1,\ldots,n,\text{ }i\neq k,l
\]
or%
\[
\frac{w_{l}}{d_{l}}\geq\frac{w_{i}}{d_{i}}\geq\frac{w_{k}}{d_{k}}\geq
\frac{w_{l}}{d_{l}}x,\text{ for }i=1,\ldots,n,\text{ }i\neq k,l.
\]

\end{theorem}

\bigskip Any matrix $A\in\mathcal{PC}_{3}$ is a simple perturbed consistent
matrix. Thus, all the efficient vectors for $A\in\mathcal{PC}_{3}$ are
obtained from Theorem \ref{tmain} (in fact, they can be obtained by a direct
inspection of all the strongly connected digraphs $G(A,w)$ with $w\in
\mathcal{V}_{3}$).

\begin{corollary}
\label{c3por3}Let
\[
\left[
\begin{array}
[c]{ccc}%
1 & a_{12} & a_{13}\\
\frac{1}{a_{12}} & 1 & a_{23}\\
\frac{1}{a_{13}} & \frac{1}{a_{23}} & 1
\end{array}
\right]  .
\]
Then, $w\in\mathcal{V}_{3}$ is efficient for $A$ if and only if%
\[
a_{23}w_{3}\leq w_{2}\leq\frac{w_{1}}{a_{12}}\leq\frac{a_{13}}{a_{12}}%
w_{3}\text{\qquad or \qquad}a_{23}w_{3}\geq w_{2}\geq\frac{w_{1}}{a_{12}}%
\geq\frac{a_{13}}{a_{12}}w_{3}.
\]

\end{corollary}

\begin{proof}
We have
\[
A=DS_{3,1,3}\left(  \frac{a_{13}}{a_{12}a_{23}}\right)  D^{-1},
\]
with
\[
D=\left[
\begin{array}
[c]{ccc}%
a_{12} & 0 & 0\\
0 & 1 & 0\\
0 & 0 & \frac{1}{a_{23}}%
\end{array}
\right]  \text{\qquad and\qquad}S_{3,1,3}\left(  \frac{a_{13}}{a_{12}a_{23}%
}\right)  =\left[
\begin{array}
[c]{ccc}%
1 & 1 & \frac{a_{13}}{a_{12}a_{23}}\\
1 & 1 & 1\\
\frac{a_{12}a_{23}}{a_{13}} & 1 & 1
\end{array}
\right]  .
\]
Now the claim is an immediate consequence of Theorem \ref{tmain}.
\end{proof}

\bigskip

For our inductive construction of all the efficient vectors of a reciprocal
matrix, we will use the next important result, presented in \cite{FJ1} (see
also \cite{CFF}), which gives necessary and sufficient conditions for an
efficient vector for a matrix $A\in\mathcal{PC}_{n}$ to be an extension of an
efficient vector for an $(n-1)$-by-$(n-1)$ principal submatrix of $A.$

\begin{theorem}
\cite{FJ1}\label{thext} Let $A\in\mathcal{PC}_{n},$ $w\in\mathcal{V}_{n}$ and
$k\in\{1,\ldots,n\}$. Suppose that $w(k)\in\mathcal{E}(A(k)).$ Then,
$w\in\mathcal{E}(A)$ if and only if
\[
\min_{1\leq i\leq n,\text{ }i\neq k}\frac{w_{i}}{a_{ik}}\leq w_{k}\leq
\max_{1\leq i\leq n,\text{ }i\neq k}\frac{w_{i}}{a_{ik}}.
\]

\end{theorem}

\section{Subgraphs of a strongly connected semi-complete digraph \label{s4}}

In this section we notice an important property of semi-complete digraphs. We
use the following notation: by $G=G_{n,E}$ we mean the digraph with vertex set
$V=\{1,\ldots,n\}$ and edge set $E$; $G(i),$ $i=1,\ldots,n,$ is the subgraph
of $G$ induced by the set of vertices $V\backslash\{i\};$ $R(G)$ is the
adjacency matrix of $G$. Note that $R(G)(i)$ is the adjacency matrix of
$G(i).$

The following conditions are equivalent \cite{HJ}: $G$ is strongly connected;
$R(G)$ is irreducible; $(I_{n}+R(G))^{n-1}$ is positive.

\bigskip Here, we state that, if $n\geq4$ and $G$ is a strongly connected
semi-complete digraph, then there are at least two strongly connected induced
subgraphs of $G$ with $n-1$ vertices. However, we first observe that, if
$n=3,$ it may happen that all induced subgraphs of $G$ with $2$ vertices are
not strongly connected. In fact, this always happens if $i\rightarrow j$ being
an edge in $G$ implies that $j\rightarrow i$ is not an edge in $G,$
$i,j\in\{1,2,3\}.$

\begin{example}
Let $G$ be the strongly connected semi-complete digraph with $3$ vertices such
that%
\[
R(G)=\left[
\begin{array}
[c]{ccc}%
0 & 1 & 0\\
0 & 0 & 1\\
1 & 0 & 0
\end{array}
\right]  .
\]
The $2$-by-$2$ principal submatrices of $R(G)$ are of the form%
\[
\left[
\begin{array}
[c]{cc}%
0 & 1\\
0 & 0
\end{array}
\right]  \text{ and }\left[
\begin{array}
[c]{cc}%
0 & 0\\
1 & 0
\end{array}
\right]  ,
\]
which are reducible. This implies that none of the three induced subgraphs of
$G$ with $2$ vertices is strongly connected.
\end{example}

The following result can be found in \cite[Theorem 1.5.3]{Jensen}.

\begin{lemma}
\cite{Jensen}\label{Moon}Let $G=G_{n,E}$, $n\geq3,$ be a strongly connected
semi-complete digraph. Then, each vertex of $G$ is contained in some directed
cycle of length $k$, for all $k=3,\ldots,n.$
\end{lemma}

\bigskip A tournament is a semi-complete digraph with exactly one directed
edge between any two vertices \cite{Moonbook}. Trivially, a semi-complete
digraph is a tournament with perhaps some additional directed edges. We note
in passing that, from Lemma \ref{Moon}, it follows that every strongly
connected semi-complete digraph contains a full directed cycle and, therefore,
contains a strongly connected tournament.

\bigskip

As a consequence of the previous lemma, we then get the following important fact.

\begin{theorem}
\label{lgraph}Let $G=G_{n,E}$, $n\geq4,$ be a strongly connected semi-complete
digraph. Then, there are at least two strongly connected induced subgraphs of
$G$ with $n-1$ vertices.
\end{theorem}

\begin{proof}
By Lemma \ref{Moon} there is a cycle $\mathcal{C}$ in $G$ of length $n-1$. Let
$i$ be the vertex not contained in $\mathcal{C}.$ Again, by the lemma, there
is a cycle $\mathcal{C}^{\prime}$ in $G$ of length $n-1$ containing vertex
$i.$ If $j$ is the vertex not contained in $C^{\prime},$ we have that $G(i)$
and $G(j)$ are strongly connected, as both contain full cycles.
\end{proof}

\bigskip

It appears common that a strongly connected semi-complete digraph with $n$
vertices has more than two strongly connected induced subgraphs with $n-1$
vertices, but there may be only two, as illustrated next.

\begin{example}
\label{ex1}Let%
\[
R=\left[
\begin{array}
[c]{cccccc}%
0 & 1 & 1 & 0 & 1 & 1\\
0 & 0 & 1 & 0 & 1 & 0\\
0 & 0 & 0 & 1 & 1 & 1\\
1 & 1 & 0 & 0 & 1 & 1\\
0 & 0 & 0 & 0 & 0 & 1\\
0 & 1 & 0 & 0 & 0 & 0
\end{array}
\right]  .
\]
The matrix $R$ is the adjacency matrix of a semi-complete strongly connected
digraph $G$. We have $R(i)$ irreducible if and only if $i=1$ or $i=5,$ that
is, the only induced subgraphs of $G$ with $5$ vertices that are strongly
connected are $G(1)$ and $G(5).$
\end{example}

\section{Recursive construction of all the efficient vectors for a reciprocal
matrix\label{s44}}

Here and throughout, we use the following notation, which plays an important
role in the presentation of our results. For $A\in\mathcal{PC}_{n}$ and
$i\in\{1,\ldots,n\},$ we denote by $\mathcal{E}(A;i)$ the set of vectors $w$
efficient for $A$ and such that $w(i)$ is efficient for $A(i),$ that is,
$w\in\mathcal{E}(A)$ and $w(i)\in\mathcal{E}(A(i)).$

\bigskip

From Theorem \ref{blanq} and Theorem \ref{lgraph}, we obtain the following
result, which allows to construct all efficient vectors $w$ for a given matrix
$A\in\mathcal{PC}_{n}.$ Note that $G(A,w)(i)=G(A(i),w(i))$ for $i\in
\{1,\ldots,n\}.$

\begin{theorem}
\label{tnbyn}Let $A\in\mathcal{PC}_{n}$, $n\geq4,$ and $w\in\mathcal{E}(A).$
Then, there are $i,j\in\{1,\ldots,n\},$ with $i\neq j,$ such that
$w(i)\in\mathcal{E}(A(i))$ and $w(j)\in\mathcal{E}(A(j)),$ that is,
$w\in\mathcal{E}(A;i)\cap\mathcal{E}(A;j).$
\end{theorem}

As noted before, if $A\in\mathcal{PC}_{n},$ $G=G(A,w)$ is a semi-complete
digraph. The converse also holds. If $G$ is a semi-complete digraph with $n$
vertices, there is a matrix $A=[a_{ij}]\in\mathcal{PC}_{n}$ and a vector
$w\in\mathcal{V}_{n}$ such that $G=G(A,w).$ Just take $w=\mathbf{e}_{n},$ the
n-vector with all entries $1$, and $A$ such that $a_{ij}<1$ if $i\rightarrow
j$ is an edge in $G$ and $j\rightarrow i$ is not an edge in $G,$ and
$a_{ij}=1$ if both $i\rightarrow j$ and $j\rightarrow i$ are edges in $G$.
Thus, based on Example \ref{ex1}, we give a matrix $A\in\mathcal{PC}_{6}$ and
a vector $w\in\mathcal{E}(A)$ such that $w(i)\in\mathcal{E}(A(i))$ for exactly
two distinct $i\in\{1,\ldots,n\}.$

\begin{example}
Let $w=\mathbf{e}_{6}$ and
\[
A=\left[
\begin{array}
[c]{cccccc}%
1 & 0.5 & 0.5 & 2 & 0.5 & 0.5\\
2 & 1 & 0.5 & 2 & 0.5 & 2\\
2 & 2 & 1 & 0.5 & 0.5 & 0.5\\
0.5 & 0.5 & 2 & 1 & 0.5 & 0.5\\
2 & 2 & 2 & 2 & 1 & 0.5\\
2 & 0.5 & 2 & 2 & 2 & 1
\end{array}
\right]  .
\]
The adjacency matrix of $G(A,w)$ is the matrix $R$ in Example \ref{ex1}. Thus,
$w\in\mathcal{E}(A),$ and $w(i)\in\mathcal{E}(A(i))$ if and only if
$i\in\{1,5\}.$
\end{example}

We next give an example of a family of matrices $A\in\mathcal{PC}_{5}$ and a
vector $w\in\mathcal{E}(A)$ such that $w(i)\in\mathcal{E}(A(i))$ for every
$i\in\{1,\ldots,5\}.$

\begin{example}
Let $x_{i}>0,$ $i=1,2,3.$ Let $w=\mathbf{e}_{5}$ and
\[
A=\left[
\begin{array}
[c]{ccccc}%
1 & 1 & 1 & x_{1} & x_{2}\\
1 & 1 & 1 & 1 & x_{3}\\
1 & 1 & 1 & 1 & 1\\
\frac{1}{x_{1}} & 1 & 1 & 1 & 1\\
\frac{1}{x_{2}} & \frac{1}{x_{3}} & 1 & 1 & 1
\end{array}
\right]  \in\mathcal{PC}_{5}.
\]
Then, $w\in\mathcal{E}(A)$ and $w(i)\in\mathcal{E}(A(i))$ for any
$i\in\{1,\ldots,5\}.$ The same happens, if instead of $A$ and $w,$ we take
$DAD^{-1}$ and $Dw,$ for any $5$-by-$5$ positive diagonal matrix $D.$
\end{example}

From Theorem \ref{tnbyn}, we obtain the following.

\begin{corollary}
\label{cunion}Let $A\in\mathcal{PC}_{n}$, $n\geq4.$ Then,
\[
\mathcal{E}(A)=%
%TCIMACRO{\dbigcup \limits_{i=1}^{n}}%
%BeginExpansion
{\displaystyle\bigcup\limits_{i=1}^{n}}
%EndExpansion
\mathcal{E}(A;i)=%
%TCIMACRO{\dbigcup \limits_{i=1,\text{ }i\neq p}^{n}}%
%BeginExpansion
{\displaystyle\bigcup\limits_{i=1,\text{ }i\neq p}^{n}}
%EndExpansion
\mathcal{E}(A;i),
\]
for any $p\in\{1,\ldots,n\}.$
\end{corollary}

Applying Theorem \ref{tnbyn} recursively, we get the following.

\begin{corollary}
Let $A\in\mathcal{PC}_{n}$, $n\geq4.$ If $w\in\mathcal{E}(A),$ then, for any
$p\in\{3,\ldots,n-1\},$ there are at least two principal submatrices of $A$,
$A[K_{1}],$ $A[K_{2}]\in\mathcal{PC}_{p}$ (indexed by two distinct subsets
$K_{1},K_{2}\subset\{1,\ldots,n\})$, such that $w[K_{1}]\in\mathcal{E}%
(A[K_{1}])$ and $w[K_{2}]\in\mathcal{E}(A[K_{2}]).$
\end{corollary}

\bigskip

If $A\in\mathcal{PC}_{3},$ the efficient vectors for $A$ are described in
Corollary \ref{c3por3}. If $A\in\mathcal{PC}_{n}$, $\ $with $n\geq4,$ based on
Corollary \ref{cunion}, we obtain an algorithm to construct all efficient
vectors for $A$. In fact, by applying recursively Corollary \ref{cunion}, we
have that any efficient vector for $A$ has an $(n-1)$-subvector efficient for
the corresponding principal submatrix of $A$ in $\mathcal{PC}_{n-1},$ which in
turn has a subvector efficient for the corresponding principal submatrix of
$A$ in $\mathcal{PC}_{n-2},$ etc... So, starting with $p=3,$ for each
$p$-by-$p$ principal submatrix $B$ of $A,$ we extend all its efficient vectors
to efficient vectors for the $(p+1)$-by-$(p+1)$ principal submatrices of $A$
that have $B$ as a principal submatrix. Applying this procedure successively,
we obtain $\mathcal{E}(A).$ Next we formalize this construction.

\bigskip

For $k=1,\ldots,n,$ denote by $\mathcal{C}_{A,k}$ the collection of all
$k$-by-$k$ principal submatrices of $A.$ Note that there are $C_{k}^{n}$
matrices in $\mathcal{C}_{A,k}$ (of course some may be equal) and
$\mathcal{C}_{A,n}=\{A\}.$

\bigskip

\textbf{Algorithm}

\bigskip

Let $A\in\mathcal{PC}_{n}$ and $p=3$

for any $B\in\mathcal{C}_{A,p},$

$\qquad$use Corollary \ref{c3por3} to obtain all the efficient vectors for
$B,$

$\qquad$giving the set of vectors $\mathcal{E}(B)$

endfor

for $k=p+1$ to $n$

$\qquad$for any $B\in\mathcal{C}_{A,k}$

$\qquad\qquad$let $\mathcal{E}(B)=\emptyset$

$\qquad\qquad$for any $i\in\{1,\ldots,k\}$ (except eventually one)

$\qquad\qquad\qquad$extend all the efficient vectors for $B(i)$ (in
$\mathcal{E}(B(i))$)

$\qquad\qquad\qquad$to efficient vectors for $B$ using Theorem \ref{thext},

$\qquad\qquad\qquad$giving the set of vectors $\mathcal{E}(B;i)$

$\qquad\qquad\qquad$let $\mathcal{E}(B)=\mathcal{E}(B)\cup\mathcal{E}(B;i)$

$\qquad\qquad$endfor

$\qquad$endfor

endfor

\bigskip

For $k=n,$ we get $\mathcal{E}(A),$ the set of all efficient vectors for $A.$

\bigskip

We note that, if all efficient vectors for the matrices in $\mathcal{C}_{A,p}$
are known for some $p>3,$ we may start our algorithm with this value of $p,$
skipping the first step in the algorithm concerning the determination of
$\mathcal{E}(B)$ for $B\in\mathcal{C}_{A,3}.$

We also observe that, when $B\in\mathcal{C}_{A,k},$ $3\leq k\leq n,$ is a
simple perturbed consistent matrix, Theorem \ref{tmain} is helpful to
construct directly the efficient vectors for $B$.

\section{$\mathcal{E}(A)$ is connected\label{s5}}

For $A\in\mathcal{PC}_{n},$ the precise "geometric" nature of $\mathcal{E}(A)$
is unknown. If $A\in\mathcal{PC}_{2}$ then $A$ is consistent and, thus,
$\mathcal{E}(A)$ consists of positive multiples of a single vector. If $n=3,$
$\mathcal{\ }A$ is a simple perturbed consistent matrix and, by Corollary
\ref{c3por3}, the efficient vectors for $A\ $are defined by a finite system of
linear inequalities in the vector entries. Thus, $\mathcal{E}(A)$ is convex
for $n\leq3,$ but it is not in general. Since any column of $A$ is efficient
for $A$ \cite{FJ1} and the Perron vector of $A$ is a positive linear
combination of the columns of $A,$ when $\mathcal{E}(A)$ is convex, the Perron
vector is efficient for $A$. For each $n\geq4,$ examples of matrices
$A\in\mathcal{PC}_{n}$ with inefficient Perron vector are known
\cite{blanq2006, bozoki2014}, in which cases $\mathcal{E}(A)$ is not convex.

In \cite{blanq2006}, an indirect existential proof of connectivity of
$\mathcal{E}(A)$ is stated, based upon a general result from optimization
\cite{carrizoza1995}. This proof is not enlightening about the explicit nature
of connectivity. We have found connectivity to be a helpful fact, and, here,
we give a constructive proof of piecewise linear (PWL) connectivity of
$\mathcal{E}(A)$ and of each subset $\mathcal{E}(A;i),$ $i=1,\ldots,n$, the
latter being new. Our approach is based on the algebraic structure we have
developed. Our overall method is inductive on $n.$

\bigskip

Let $x,y\in S\subseteq\mathbb{R}_{+}^{k}.$ We say that $x$ and $y$ are
\emph{PWL connected} in $S$ if, for some $u_{0},\ldots,u_{s}\in S,$ with
$u_{0}=x$ and $u_{s}=y,$ we have $tu_{i}+(1-t)u_{i+1}\in S$ for any
$i=0,\ldots,s-1$ and any $t\in(0,1)$. Clearly, if, for any $x,y\in S,$ $x,y$
is PWL connected then $S$ is connected. In this case, we say that $S\ $is PWL
connected. Note that any convex set is PWL connected.

We say that the vectors $\left[
\begin{array}
[c]{ccc}%
w_{1} & \cdots & w_{n-1}%
\end{array}
\right]  ^{T},\left[
\begin{array}
[c]{ccc}%
v_{1} & \cdots & v_{n-1}%
\end{array}
\right]  ^{T}\in\mathcal{V}_{n-1}$ have the same \emph{maximal} (resp.
\emph{minimal}) $n$\emph{-index} for $A\in\mathcal{PC}_{n}$ if there exists
$p\in\{1,\ldots,n-1\}$ such that
\[
\max_{1\leq i\leq n-1}\frac{w_{i}}{a_{in}}=\frac{w_{p}}{a_{pn}}\text{ and
}\max_{1\leq i\leq n-1}\frac{v_{i}}{a_{in}}=\frac{v_{p}}{a_{pn}}%
\]
(resp. $\min_{1\leq i\leq n-1}\frac{w_{i}}{a_{in}}=\frac{w_{p}}{a_{pn}}$ and
$\min_{1\leq i\leq n-1}\frac{v_{i}}{a_{in}}=\frac{v_{p}}{a_{pn}}$)$.$

\bigskip

Observe that $\max_{1\leq i\leq n-1}\frac{w_{i}}{a_{in}}$ and $\min_{1\leq
i\leq n-1}\frac{w_{i}}{a_{in}}\ $depend continuously on $w_{1},\ldots
,w_{n-1}.$

\begin{lemma}
\label{ll1}Let $A\in\mathcal{PC}_{n}$ and $x,y\in\mathcal{E}(A(n)).$ Let
$t\in(0,1)$ and suppose that $tx+(1-t)y\in\mathcal{E}(A(n)).$ Then, the
following statements are equivalent:

\begin{enumerate}
\item $tw+(1-t)v\in\mathcal{E}(A),$ for all $w,v\in\mathcal{E}(A)$ with
$w(n)=x$ and $v(n)=y;$

\item $x$ and $y$ have the same maximal and the same minimal $n$-indices for
$A.$
\end{enumerate}

\begin{proof}
Suppose that $1.$ holds and $w(n)$ and $v(n)$ do not have the same maximal
$n$-index. Suppose that $\max_{1\leq i\leq n-1}\frac{w_{i}}{a_{in}}%
=\frac{w_{p}}{a_{pn}}$ and $\max_{1\leq i\leq n-1}\frac{v_{i}}{a_{in}}%
=\frac{v_{q}}{a_{qn}}\neq\frac{v_{p}}{a_{pn}}.$ Let $w_{n}=\frac{w_{p}}%
{a_{pn}}$ and $v_{n}=\frac{v_{q}}{a_{qn}}.$ By Theorem \ref{thext},
$w,v\in\mathcal{E}(A).$ We will show that $tw+(1-t)v\notin\mathcal{E}(A),$
contradicting the hypothesis. Since $tw(n)+(1-t)v(n)\in\mathcal{E}(A(n)),$ by
Theorem \ref{thext}, it is enough to note that
\begin{align*}
tw_{n}+(1-t)v_{n}  &  =t\frac{w_{p}}{a_{pn}}+(1-t)\frac{v_{q}}{a_{qn}}\\
&  =\max_{1\leq i\leq n-1}\frac{tw_{i}}{a_{in}}+\max_{1\leq i\leq n-1}%
\frac{(1-t)v_{i}}{a_{in}}\\
&  >\max_{1\leq i\leq n-1}\frac{tw_{i}+(1-t)v_{i}}{a_{in}}.
\end{align*}
Similar arguments can be applied if $w(n)$ and $v(n)$ do not have the same
minimal $n$-index.

Suppose that $2.$ holds and let $w,v\in\mathcal{E}(A)$ with $w(n)=x$ and
$v(n)=y.$ Let
\[
\max_{1\leq i\leq n-1}\frac{w_{i}}{a_{in}}=\frac{w_{p}}{a_{pn}}\text{ and
}\max_{1\leq i\leq n-1}\frac{v_{i}}{a_{in}}=\frac{v_{p}}{a_{pn}},
\]
and%
\[
\min_{1\leq i\leq n-1}\frac{w_{i}}{a_{in}}=\frac{w_{q}}{a_{qn}}\text{ and
}\min_{1\leq i\leq n-1}\frac{v_{i}}{a_{in}}=\frac{v_{q}}{a_{qn}}.
\]
Then,%
\[
\max_{1\leq i\leq n-1}\frac{tw_{i}+(1-t)v_{i}}{a_{in}}=\frac{tw_{p}%
+(1-t)v_{p}}{a_{pn}}=t\max_{1\leq i\leq n-1}\frac{w_{i}}{a_{in}}%
+(1-t)\max_{1\leq i\leq n-1}\frac{v_{i}}{a_{in}}%
\]
and%
\[
\min_{1\leq i\leq n-1}\frac{tw_{i}+(1-t)v_{i}}{a_{in}}=\frac{tw_{q}%
+(1-t)v_{q}}{a_{qn}}=t\min_{1\leq i\leq n-1}\frac{w_{i}}{a_{in}}%
+(1-t)\min_{1\leq i\leq n-1}\frac{v_{i}}{a_{in}}.
\]
Since $w,v\in\mathcal{E}(A;n),$ by Theorem \ref{thext},
\[
\min_{1\leq i\leq n-1}\frac{w_{i}}{a_{in}}\leq w_{n}\leq\max_{1\leq i\leq
n-1}\frac{w_{i}}{a_{in}}\text{ and }\min_{1\leq i\leq n-1}\frac{v_{i}}{a_{in}%
}\leq v_{n}\leq\max_{1\leq i\leq n-1}\frac{v_{i}}{a_{in}}.
\]
Thus,%
\[
\min_{1\leq i\leq n-1}\frac{tw_{i}+(1-t)v_{i}}{a_{in}}\leq tw_{n}%
+(1-t)v_{n}\leq\max_{1\leq i\leq n-1}\frac{tw_{i}+(1-t)v_{i}}{a_{in}}.
\]
By Theorem \ref{thext}, $tw+(1-t)v\in\mathcal{E}(A;n)$.
\end{proof}
\end{lemma}

The following lemma will be helpful in the proof of Lemma \ref{l55}.

\begin{lemma}
\label{lcont}Let $u_{i}(t),$ $\ i=1,\ldots,k,$ be linear functions of
$\ t\in\lbrack0,1].$ Let \ $m_{\max}(t)=\max\{u_{1}(t),\ldots,u_{k}(t)\}$ and
$m_{\min}(t)=\min\{u_{1}(t),\ldots,u_{k}(t)\}).$ Then, $m_{\max}(t)$ and
$m_{\min}(t)$ are piecewise linear functions of $t$ and there is\ a finite
partition of $[0,1],$ say
\begin{equation}
\ 0=t_{0}<t_{1}<\cdots<t_{s}=1, \label{part}%
\end{equation}
and indices \ $p_{0},\ldots,p_{s-1},q_{0},\ldots,q_{s-1}\in\{1,\ldots,k\}$
such that
\[
m_{\max}(t)=u_{p_{i}}(t)\text{ and }m_{\min}(t)=u_{q_{i}}(t)
\]
for all $t\in\lbrack t_{i},t_{i+1}]$ and all \ $i=0,\ldots,s-1$. In
particular, if $s>1$, then $u_{p_{i}}(t_{i+1})=u_{p_{i+1}}(t_{i+1})$ and
$u_{q_{i}}(t_{i+1})=u_{q_{i+1}}(t_{i+1}).$
\end{lemma}

\begin{proof}
Since any two distinct line segments intersect at most once, and the
$u_{j}(t),$ $t\in\lbrack0,1],$ $j=1,\ldots,k,$ constitute $k$ line segments,
there is a partition as in (\ref{part}) such that no two distinct line
segments intersect in the open interval $(t_{i},t_{i+1}).$ Then, in each of
these intervals, which can be assumed closed by continuity, there is one line
segment bounding below all the line segments for that interval, and one
bounding above.
\end{proof}

\begin{lemma}
\label{l55}Let $A\in\mathcal{PC}_{n},$ $n\geq2,$ and $i\in\{1,\ldots,n\}.$
Suppose that $\mathcal{E}(A(i))$ is PWL connected. Then, $\mathcal{E}(A;i)$ is
PWL connected.
\end{lemma}

\begin{proof}
By Lemma \ref{lsim}, we assume that $i=n.$ Since $\mathcal{E}(A(n))$ is PWL
connected, it is enough to show that, if $w(n)$ and $v(n)$ lie in a line
segment contained in $\mathcal{E}(A(n)),$ then all extensions $w,v\in
\mathcal{E}(A)$ of $w(n)$ and $v(n)$ are PWL connected in $\mathcal{E}(A),$
that is, there is a piecewise line segment in $\mathcal{E}(A)$ connecting $w$
and $v$. Suppose that, for every $t\in\lbrack0,1],$
\begin{equation}
z(n,t):=tw(n)+(1-t)v(n)\in\mathcal{E}(A(n)). \label{znt}%
\end{equation}
Using Lemma \ref{lcont} (taking $k=n-1$ and $u_{i}(t)=\frac{tw_{i}+(1-t)v_{i}%
}{a_{in}}$, in which $w_{i}$ and $v_{i}$ are the $i$th entry of $w(n)$ and
$v(n),$ respectively), there is a partition of $[0,1]$ as in (\ref{part}) such
that $z(n,t)$ has the same maximal and the same minimal $n$-indices for $A$
for any $t\in\lbrack t_{i},t_{i+1}],$ $i=0,\ldots,s-1.$ We have that
$t^{\prime}z(n,t_{i})+(1-t^{\prime})z(n,t_{i+1})\in\mathcal{E}(A(n))$ for any
$t^{\prime}\in\lbrack0,1],$ as this corresponds to a line segment contained in
the one defined by (\ref{znt}) for $t\in\lbrack0,1].$ Let $z(t_{i}),$
$z(t_{i+1})\in\mathcal{E}(A;n)$ be extensions of $z(n,t_{i})$ and
$z(n,t_{i+1})$, with $z(t_{0})=v$ and $z(t_{s})=w.$ By Lemma \ref{ll1},
$t^{\prime}z(t_{i})+(1-t^{\prime})z(t_{i+1})\in\mathcal{E}(A;n)$ for any
$t^{\prime}\in\lbrack0,1],$ $i=0,\ldots,s-1,$ that is, each line segment
connecting $z(t_{i})$ and $z(t_{i+1})$ lies in $\mathcal{E}(A;n).$ Since
$z(t_{s})=w$ and $z(t_{0})=v,$ the claim follows.
\end{proof}

\bigskip

The following result was noted in \cite{FJ5}.

\begin{lemma}
\label{thind}Let $A\in\mathcal{PC}_{n},$ with $n>2$, and $w\in\mathcal{V}_{n}%
$. If there are $i,j\in\{1,\ldots,n\},$ with $i\neq j,$ such that $w(i)$ is
efficient for $A(i)$ and $w(j)$ is efficient for $A(j),$ then $w$ is efficient
for $A.$
\end{lemma}

Using the previous fact, we then get the following.

\begin{lemma}
\label{l1}Let $A\in\mathcal{PC}_{n},$ $n\geq4,$ and $i,j\in\{1,\ldots,n\},$
$i\neq j.$ Then there is $w\in\mathcal{E}(A)$ such that $w(i)\in
\mathcal{E}(A(i))$ and $w(j)\in\mathcal{E}(A(j)).$
\end{lemma}

\begin{proof}
By Lemma \ref{lsim}, we may assume, without loss of generality, that $i=n-1$
and $j=n.$ Choose $v\in\mathcal{E}(A(\{n-1,n\})).$ Let
\[
\left[
\begin{array}
[c]{cc}%
v^{T} & w_{n}%
\end{array}
\right]  ^{T}\qquad\text{and}\qquad\left[
\begin{array}
[c]{cc}%
v^{T} & w_{n-1}%
\end{array}
\right]  ^{T}%
\]
be extensions of $v$ to efficient vectors for $A(n-1)$ and $A(n),$
respectively, according to Theorem \ref{thext}. Then, by Lemma \ref{thind},
\[
w=\left[
\begin{array}
[c]{ccc}%
v^{T} & w_{n-1} & w_{n}%
\end{array}
\right]  ^{T}%
\]
is efficient for $A.$
\end{proof}

\bigskip

We now state the main results of this section.

\begin{theorem}
\label{mainconnected}Let $A\in\mathcal{PC}_{n}.$ Then, $\mathcal{E}(A)$ is PWL connected.
\end{theorem}

\begin{proof}
The proof is by induction on $n.$ For $n\leq3$ it has already been noticed
that $\mathcal{E}(A)$ is convex, and, thus, is PWL connected. Suppose that
$n>3$ and that, by the induction hypothesis, $\mathcal{E}(A(i))$ is PWL
connected, for all $i\in\{1,\ldots,n\}.$

Let $u,v\in\mathcal{E}(A).$ By Theorem \ref{tnbyn}, there are $i_{1}%
,i_{2},j_{1},j_{2}\in\{1,\ldots,n\},$ with $i_{1}\neq i_{2}$ and $j_{1}\neq
j_{2},$ such that $u\in\mathcal{E}(A;i_{1})\cap\mathcal{E}(A;i_{2})$ and
$v\in\mathcal{E}(A;j_{1})\cap\mathcal{E}(A;j_{2}).$

If there is $s\in\left\{  i_{1},i_{2}\right\}  \cap\left\{  j_{1}%
,j_{2}\right\}  ,$ then $u,v\in\mathcal{E}(A;s),$ which is PWL connected by
Lemma \ref{l55}.

If $\left\{  i_{1},i_{2}\right\}  \cap\left\{  j_{1},j_{2}\right\}  $ is
empty, by Lemma \ref{l1}, there is a $z\in\mathcal{E}(A;i_{1})\cap
\mathcal{E}(A;j_{1})$. Then, $u,z\in\mathcal{E}(A;i_{1})$ and $z,v\in
\mathcal{E}(A;j_{1}),$ and $\mathcal{E}(A;i_{1})$ and $\mathcal{E}(A;j_{1})$
are PWL connected by Lemma \ref{l55}.

Then $u,v\in\mathcal{E}(A)$ are PWL connected in $\mathcal{E}(A)$ and the
proof is complete.
\end{proof}

\begin{corollary}
\label{l2}Let $A\in\mathcal{PC}_{n},$ $n\geq2.$ Then, for any $i\in
\{1,\ldots,n\},$ $\mathcal{E}(A;i)$ is PWL connected.
\end{corollary}

\begin{proof}
By Theorem \ref{mainconnected}, $\mathcal{E}(A(i))$ is PWL connected,
implying, by Lemma \ref{l55}, that $\mathcal{E}(A;i)$ is PWL connected.
\end{proof}

\bigskip

Though, for $n>3,$ $\mathcal{E}(A;i)$ is connected, it may happen that it is
not convex, as the next example shows.

\begin{example}
Consider the matrix%
\[
A=\left[
\begin{array}
[c]{cccc}%
1 & 1 & 5 & 4\\
1 & 1 & 1 & 1\\
\frac{1}{5} & 1 & 1 & 6\\
\frac{1}{4} & 1 & \frac{1}{6} & 1
\end{array}
\right]  .
\]
The set $\mathcal{E}(A(4))$ is convex as $A(4)\in\mathcal{PC}_{3},$ and, by
Corollary \ref{c3por3}, $\left[
\begin{array}
[c]{ccc}%
w_{1} & w_{2} & w_{3}%
\end{array}
\right]  ^{T}\in\mathcal{E}(A(4))$ if and only if%
\[
w_{3}\leq w_{2}\leq w_{1}\leq5w_{3}.
\]
Also, by Theorem \ref{thext}, $w\in\mathcal{E}(A;4)$ if and only if%
\[
\min\left\{  \frac{w_{1}}{4},w_{2},\frac{w_{3}}{6}\right\}  \leq w_{4}\leq
\max\left\{  \frac{w_{1}}{4},w_{2},\frac{w_{3}}{6}\right\}  .
\]
For example,
\[
u(4):=\left[
\begin{array}
[c]{ccc}%
5 & 1 & 1
\end{array}
\right]  ^{T},\text{ }v(4):=\left[
\begin{array}
[c]{ccc}%
3 & 1 & 1
\end{array}
\right]  ^{T}\in\mathcal{E}(A(4)).
\]
and, for any $t\in(0,1),$
\[
tu(4)+(1-t)v(4)\in\mathcal{E}(A(4)),
\]
as $\mathcal{E}(A(4))$ is convex.
\end{example}

By Theorem \ref{thext},
\[
u=\left[
\begin{array}
[c]{cccc}%
5 & 1 & 1 & \frac{5}{4}%
\end{array}
\right]  ^{T},\text{ }v=\left[
\begin{array}
[c]{cccc}%
3 & 1 & 1 & 1
\end{array}
\right]  ^{T}\in\mathcal{E}(A;4),
\]
as%
\begin{equation}
\min\left\{  \frac{5}{4},1,\frac{1}{6}\right\}  \leq\frac{5}{4}\leq
\max\left\{  \frac{5}{4},1,\frac{1}{6}\right\}  , \label{f1}%
\end{equation}
and%
\begin{equation}
\min\left\{  \frac{3}{4},1,\frac{1}{6}\right\}  \leq1\leq\max\left\{  \frac
{3}{4},1,\frac{1}{6}\right\}  . \label{f2}%
\end{equation}
Note that $u$ and $v$ are extensions of $u(4)$ and $v(4)$, respectively.
However, by Theorem \ref{thext}, for any $t\in(0,1)$ sufficiently close to
$1,$
\[
tu+(1-t)v\notin\mathcal{E}(A)
\]
as
\[
\frac{5}{4}t+(1-t)>\frac{5}{4}t+(1-t)\frac{3}{4}=\max\left\{  \frac{5}%
{4}t+(1-t)\frac{3}{4},1,\frac{1}{6}\right\}  .
\]
The reason why this happens is that $u(4)$ and $v(4)$ do not have the same
maximal $n$-index for $A$ (Lemma \ref{ll1}).

\section{Efficient vectors for a $4$-by-$4$ reciprocal matrix\label{sec4x4}}

Based on Corollary \ref{cunion}, we construct all the efficient vectors for an
arbitrary matrix $A\in\mathcal{PC}_{4}.$ Using Corollary \ref{c3por3}, we
first obtain the set of efficient vectors $\mathcal{E}(A(i))$ for each
$3$-by-$3$ principal submatrix $A(i)$ of $A$ ($i=1,2,3,4$). Then,
$\mathcal{E}(A)$ is the union of the sets $\mathcal{E}(A;i)$ of efficient
vectors for $A$ that extend the efficient vectors for $A(i),$ according to
Theorem \ref{thext}.

Let%
\[
A=\left[
\begin{array}
[c]{cccc}%
1 & a_{12} & a_{13} & a_{14}\\
\frac{1}{a_{12}} & 1 & a_{23} & a_{24}\\
\frac{1}{a_{13}} & \frac{1}{a_{23}} & 1 & a_{34}\\
\frac{1}{a_{14}} & \frac{1}{a_{24}} & \frac{1}{a_{34}} & 1
\end{array}
\right]  \in\mathcal{PC}_{4}.
\]
The four $3$-by-$3$ principal submatrices of $A\ $are%
\begin{align*}
A(1)  &  =\left[
\begin{array}
[c]{ccc}%
1 & a_{23} & a_{24}\\
\frac{1}{a_{23}} & 1 & a_{34}\\
\frac{1}{a_{24}} & \frac{1}{a_{34}} & 1
\end{array}
\right]  ,\text{ }A(2)=\left[
\begin{array}
[c]{ccc}%
1 & a_{13} & a_{14}\\
\frac{1}{a_{13}} & 1 & a_{34}\\
\frac{1}{a_{14}} & \frac{1}{a_{34}} & 1
\end{array}
\right]  ,\text{ }\\
& \\
A(3)  &  =\left[
\begin{array}
[c]{ccc}%
1 & a_{12} & a_{14}\\
\frac{1}{a_{12}} & 1 & a_{24}\\
\frac{1}{a_{14}} & \frac{1}{a_{24}} & 1
\end{array}
\right]  ,\text{ }A(4)=\left[
\begin{array}
[c]{ccc}%
1 & a_{12} & a_{13}\\
\frac{1}{a_{12}} & 1 & a_{23}\\
\frac{1}{a_{13}} & \frac{1}{a_{23}} & 1
\end{array}
\right]  .
\end{align*}
Denote $w=\left[
\begin{array}
[c]{cccc}%
w_{1} & w_{2} & w_{3} & w_{4}%
\end{array}
\right]  ^{T}\in\mathcal{V}_{4}$ and let $w(i)$ be the vector obtained from
$w$ by deleting the $i$th entry. By Corollary \ref{c3por3}, we have%
\[
\mathcal{E}(A(1))=\left\{  w(1):a_{34}w_{4}\leq w_{3}\leq\frac{w_{2}}{a_{23}%
}\leq\frac{a_{24}}{a_{23}}w_{4}\text{ or }a_{34}w_{4}\geq w_{3}\geq\frac
{w_{2}}{a_{23}}\geq\frac{a_{24}}{a_{23}}w_{4}\right\}  ,
\]%
\[
\mathcal{E}(A(2))=\left\{  w(2):a_{34}w_{4}\leq w_{3}\leq\frac{w_{1}}{a_{13}%
}\leq\frac{a_{14}}{a_{13}}w_{4}\text{ or }a_{34}w_{4}\geq w_{3}\geq\frac
{w_{1}}{a_{13}}\geq\frac{a_{14}}{a_{13}}w_{4}\right\}  ,
\]%
\[
\mathcal{E}(A(3))=\left\{  w(3):a_{24}w_{4}\leq w_{2}\leq\frac{w_{1}}{a_{12}%
}\leq\frac{a_{14}}{a_{12}}w_{4}\text{ or }a_{24}w_{4}\geq w_{2}\geq\frac
{w_{1}}{a_{12}}\geq\frac{a_{14}}{a_{12}}w_{4}\right\}  ,
\]%
\[
\mathcal{E}(A(4))=\left\{  w(4):a_{23}w_{3}\leq w_{2}\leq\frac{w_{1}}{a_{12}%
}\leq\frac{a_{13}}{a_{12}}w_{3}\text{ or }a_{23}w_{3}\geq w_{2}\geq\frac
{w_{1}}{a_{12}}\geq\frac{a_{13}}{a_{12}}w_{3}\right\}  .
\]
Next, using Theorem \ref{thext}, we extend the vectors in $\mathcal{E}(A(i)),$
$i=1,2,3,4,$ to efficient vectors for $A,$ giving the set $\mathcal{E}(A;i).$
We have%
\[
\mathcal{E}(A;1)=\left\{  w:w(1)\in\mathcal{E}(A(1))\text{ and }\min\left\{
a_{12}w_{2},a_{13}w_{3},a_{14}w_{4}\right\}  \leq w_{1}\leq\max\left\{
a_{12}w_{2},a_{13}w_{3},a_{14}w_{4}\right\}  \right\}  ,
\]%
\[
\mathcal{E}(A;2)=\left\{  w:w(2)\in\mathcal{E}(A(2))\text{ and }\min\left\{
\frac{w_{1}}{a_{12}},a_{23}w_{3},a_{24}w_{4}\right\}  \leq w_{2}\leq
\max\left\{  \frac{w_{1}}{a_{12}},a_{23}w_{3},a_{24}w_{4}\right\}  \right\}
,
\]%
\[
\mathcal{E}(A;3)=\left\{  w:w(3)\in\mathcal{E}(A(3))\text{ and }\min\left\{
\frac{w_{1}}{a_{13}},\frac{w_{2}}{a_{23}},a_{34}w_{4}\right\}  \leq w_{3}%
\leq\max\left\{  \frac{w_{1}}{a_{13}},\frac{w_{2}}{a_{23}},a_{34}%
w_{4}\right\}  \right\}  ,
\]%
\[
\mathcal{E}(A;4)=\left\{  w:w(4)\in\mathcal{E}(A(4))\text{ and }\min\left\{
\frac{w_{1}}{a_{14}},\frac{w_{2}}{a_{24}},\frac{w_{3}}{a_{34}}\right\}  \leq
w_{4}\leq\max\left\{  \frac{w_{1}}{a_{14}},\frac{w_{2}}{a_{24}},\frac{w_{3}%
}{a_{34}}\right\}  \right\}  .
\]
Finally, by Corollary \ref{cunion}, we have%
\begin{align}
\mathcal{E}(A)  &  =%
%TCIMACRO{\dbigcup \limits_{i=1}^{4}}%
%BeginExpansion
{\displaystyle\bigcup\limits_{i=1}^{4}}
%EndExpansion
\mathcal{E}(A;i)=%
%TCIMACRO{\dbigcup \limits_{i=2}^{4}}%
%BeginExpansion
{\displaystyle\bigcup\limits_{i=2}^{4}}
%EndExpansion
\mathcal{E}(A;i)=%
%TCIMACRO{\dbigcup \limits_{i=1,\text{ }i\neq2}^{4}}%
%BeginExpansion
{\displaystyle\bigcup\limits_{i=1,\text{ }i\neq2}^{4}}
%EndExpansion
\mathcal{E}(A;i)\label{uni}\\
&  =%
%TCIMACRO{\dbigcup \limits_{i=1,\text{ }i\neq3}^{4}}%
%BeginExpansion
{\displaystyle\bigcup\limits_{i=1,\text{ }i\neq3}^{4}}
%EndExpansion
\mathcal{E}(A;i)=%
%TCIMACRO{\dbigcup \limits_{i=1}^{3}}%
%BeginExpansion
{\displaystyle\bigcup\limits_{i=1}^{3}}
%EndExpansion
\mathcal{E}(A;i). \label{uni2}%
\end{align}

\begin{example}
\label{ex4by4}We determine all the efficient vectors for%
\[
A=\left[
\begin{array}
[c]{cccc}%
1 & 2 & 3 & 1\\
\frac{1}{2} & 1 & \frac{1}{2} & 1\\
\frac{1}{3} & 2 & 1 & 1\\
1 & 1 & 1 & 1
\end{array}
\right]  .
\]
We have
\[
\mathcal{E}(A(1))=\left\{  \left[
\begin{array}
[c]{ccc}%
w_{2} & w_{3} & w_{4}%
\end{array}
\right]  ^{T}:w_{4}\leq w_{3}\leq2w_{2}\leq2w_{4}\right\}  ,
\]%
\[
\mathcal{E}(A(2))=\left\{  \left[
\begin{array}
[c]{ccc}%
w_{1} & w_{3} & w_{4}%
\end{array}
\right]  ^{T}:w_{4}\geq w_{3}\geq\frac{w_{1}}{3}\geq\frac{w_{4}}{3}\right\}
,
\]%
\[
\mathcal{E}(A(3))=\left\{  \left[
\begin{array}
[c]{ccc}%
w_{1} & w_{2} & w_{4}%
\end{array}
\right]  ^{T}:w_{4}\geq w_{2}\geq\frac{w_{1}}{2}\geq\frac{w_{4}}{2}\right\}
,
\]%
\[
\mathcal{E}(A(4))=\left\{  \left[
\begin{array}
[c]{ccc}%
w_{1} & w_{2} & w_{3}%
\end{array}
\right]  ^{T}:\frac{w_{3}}{2}\leq w_{2}\leq\frac{w_{1}}{2}\leq\frac{3w_{3}}%
{2}\right\}  .
\]
Then,
\[
\mathcal{E}(A;1)=\left\{  w:w(1)\in\mathcal{E}(A(1))\text{ and }\min\left\{
2w_{2},3w_{3},w_{4}\right\}  \leq w_{1}\leq\max\left\{  2w_{2},3w_{3}%
,w_{4}\right\}  \right\}  ,
\]%
\[
\mathcal{E}(A;2)=\left\{  w:w(2)\in\mathcal{E}(A(2))\text{ and }\min\left\{
\frac{w_{1}}{2},\frac{w_{3}}{2},w_{4}\right\}  \leq w_{2}\leq\max\left\{
\frac{w_{1}}{2},\frac{w_{3}}{2},w_{4}\right\}  \right\}  ,
\]%
\[
\mathcal{E}(A;3)=\left\{  w:w(3)\in\mathcal{E}(A(3))\text{ and }\min\left\{
\frac{w_{1}}{3},2w_{2},w_{4}\right\}  \leq w_{3}\leq\max\left\{  \frac{w_{1}%
}{3},2w_{2},w_{4}\right\}  \right\}  ,
\]%
\[
\mathcal{E}(A;4)=\left\{  w:w(4)\in\mathcal{E}(A(4))\text{ and }\min\left\{
w_{1},w_{2},w_{3}\right\}  \leq w_{4}\leq\max\left\{  w_{1},w_{2}%
,w_{3}\right\}  \right\}  .
\]

\medskip

The sets $\mathcal{E}(A;i)$ can be simplified to:%
\[
\mathcal{E}(A;1)=\left\{  w:w(1)\in\mathcal{E}(A(1))\text{ and }w_{4}\leq
w_{1}\leq\max\left\{  2w_{2},3w_{3}\right\}  \right\}  ,
\]%
\[
\mathcal{E}(A;2)=\left\{  w:w(2)\in\mathcal{E}(A(2))\text{ and }\min\left\{
\frac{w_{1}}{2},\frac{w_{3}}{2}\right\}  \leq w_{2}\leq\max\left\{
\frac{w_{1}}{2},w_{4}\right\}  \right\}  ,
\]%
\[
\mathcal{E}(A;3)=\left\{  w:w(3)\in\mathcal{E}(A(3))\text{ and }\frac{w_{1}%
}{3}\leq w_{3}\leq2w_{2}\right\}  ,
\]%
\[
\mathcal{E}(A;4)=\left\{  w:w(4)\in\mathcal{E}(A(4))\text{ and }\min\left\{
w_{2},w_{3}\right\}  \leq w_{4}\leq w_{1}\right\}  .
\]

\medskip Then, $\mathcal{E}(A)$ is as in (\ref{uni})-(\ref{uni2})$.$ Note that
any vector in $\mathcal{E}(A)$ should be in at least two sets $\mathcal{E}%
(A;i).$

\medskip Examples of subsets of $\mathcal{E}(A)$ are%
\begin{gather*}
\left\{  \left[
\begin{array}
[c]{cccc}%
w_{1} & 4 & 6 & 5
\end{array}
\right]  ^{T}:5\leq w_{1}\leq18\right\}  \subset\mathcal{E}(A;1)\subset
\mathcal{E}(A),\\
\left\{  \left[
\begin{array}
[c]{cccc}%
15 & w_{2} & 8 & 12
\end{array}
\right]  ^{T}:4\leq w_{2}\leq12\right\}  \subset\mathcal{E}(A;2)\subset
\mathcal{E}(A),\\
\left\{  \left[
\begin{array}
[c]{cccc}%
13 & 8 & w_{3} & 12
\end{array}
\right]  ^{T}:\frac{13}{3}\leq w_{3}\leq16\right\}  \subset\mathcal{E}%
(A;3)\subset\mathcal{E}(A),\\
\left\{  \left[
\begin{array}
[c]{cccc}%
5 & 2 & 4 & w_{4}%
\end{array}
\right]  ^{T}:2\leq w_{4}\leq5\right\}  \subset\mathcal{E}(A;4)\subset
\mathcal{E}(A).
\end{gather*}

\medskip We next notice that no subset $\mathcal{E}(A;i)$ is contained in the
union of two of the remaining three subsets. In particular, $\mathcal{E}(A)$
is not the union of just two subsets $\mathcal{E}(A;i)$. For example (we write
next $(w_{1},w_{2},w_{3},w_{4})$ for $\left[
\begin{array}
[c]{cccc}%
w_{1} & w_{2} & w_{3} & w_{4}%
\end{array}
\right]  ^{T}$),
\begin{align*}
(5,4,6,5),(9,4,6,5)  &  \in\mathcal{E}(A;1),\text{ }(5,4,6,5)\notin
\mathcal{E}(A;2)\cup\mathcal{E}(A;4),\text{ }(9,4,6,5)\notin\mathcal{E}%
(A;3),\\
(15,4,8,12),\text{ }(3,2,1,2)  &  \in\mathcal{E}(A;2),\text{ }%
(15,4,8,12)\notin\mathcal{E}(A;1)\cup\mathcal{E}(A;3),\text{ }(3,2,1,2)\notin
\mathcal{E}(A;4),\\
(13,8,7,12),(13,8,16,12)  &  \in\mathcal{E}(A;3),\text{ }(13,8,7,12)\notin
\mathcal{E}(A;1)\cup\mathcal{E}(A;4),\text{ }(13,8,16,12)\notin\mathcal{E}%
(A;2),\\
(5,2,4,2),(8,2,3,4)  &  \in\mathcal{E}(A;4),\text{ }(5,2,4,2)\notin
\mathcal{E}(A;2)\cup\mathcal{E}(A;3),\text{ }(8,2,3,4)\notin\mathcal{E}(A;1).
\end{align*}

\medskip Finally, we note that the Perron eigenvector (normalized) of $A$ is%
\[
\left[
\begin{array}
[c]{cccc}%
1.\,\allowbreak5997 & 0.7018 & 0.9134 & 1
\end{array}
\right]  ^{T}\in\mathcal{E}(A;2)\cap\mathcal{E}(A;4)\allowbreak
\]
and the Hadamard geometric mean of the columns of $A$ is%
\[
\left[
\begin{array}
[c]{cccc}%
1.\,\allowbreak5651 & 0.7071 & 0.9036 & 1
\end{array}
\right]  ^{T}\in\mathcal{E}(A;2)\cap\mathcal{E}(A;4)\allowbreak.\allowbreak
\]
Neither of these vectors lies in $\mathcal{E}(A;1)\cup\mathcal{E}%
(A;3)\allowbreak.$
\end{example}

\section{Conclusions\label{s6}$\allowbreak$}

Heretofore, there has been no known way to generate all efficient vectors for
each reciprocal matrix, only constructions of particular ones. Here we show
how to generate the complete set of efficient vectors for an arbitrary
reciprocal matrix. This is facilitated by the fact that, if $w$ is an
efficient vector for an $n$-by-$n$ reciprocal matrix $A,$ with $n\geq4,$ there
are at least two $(n-1)$-by-$(n-1)$ principal submatrices of $A$ for which the
vector obtained from $w$ by deleting the corresponding entry is efficient, and
on a method to extend these subvectors to efficient vectors for $A$. Based on
our result, we show that the set of efficient vectors for a reciprocal matrix,
as well as the subsets that appear in its description, are piecewise linearly
connected. The efficient vectors for a $4$-by-$4$ reciprocal matrix are
explicitly given. Several examples illustrating the theoretical results are provided.

The given inductive method to construct efficient vectors provides a way to
produce efficient vectors in practice and plays an important role in solving
many questions concerning the study of efficient vectors for reciprocal
matrices. Among other possible applications, it can be used in the
construction of explicit classes of efficient vectors for some block-perturbed
consistent matrices and, possibly, in the study of the existence of rank reversals.

\end{document}